\documentclass{article}
\usepackage{amsmath, amssymb, amscd, amsthm, amsfonts}
\usepackage{graphicx}
\usepackage{mathrsfs}
\usepackage{hyperref}
\usepackage{color}
\usepackage{float}
\usepackage{epstopdf}
\usepackage{tikz}
\usepackage{caption}
\usepackage{subcaption}
\usepackage{booktabs}
\usetikzlibrary{arrows.meta}
\usetikzlibrary{calc}

\newcommand{\Ftwo}{\mathbb{F}_2}
\newcommand{\mchi}{\hat{\chi}}
\newcommand{\tw}{{\rm tw}}
\newcommand{\2}{\vspace{0.2cm}}

\newcommand{\comp}[1]{\overline{#1}}

\newtheorem{theorem}{Theorem}[section]
\newtheorem{proposition}[theorem]{Proposition}

\newtheorem{remark}[theorem]{Remark}
\newtheorem{corollary}[theorem]{Corollary}
\newtheorem{conjecture}[theorem]{Conjecture}
\newtheorem{claim}[theorem]{Claim}

\title{Large induced subgraphs with prescribed degree parity}
\author{Jiangdong Ai\thanks{School of Mathematical Sciences and LPMC, Nankai University. {\tt jd@nankai.edu.cn}.}
\hspace{2mm}
Qiwen Guo\thanks{Department of Computer Science, Royal Holloway University of London. {\tt gqwmath@163.com}.}
\hspace{2mm}
Gregory Gutin\thanks{Corresponding author. Department of Computer Science, Royal Holloway University of London,  and School of Mathematical Sciences and LPMC, Nankai University. {\tt g.gutin@rhul.ac.uk}.}
\hspace{2mm} Yiming Hao\thanks{School of Mathematical Sciences and LPMC, Nankai University. {\tt  
1120230031@mail.nankai.edu.cn}.}
\hspace{2mm} Anders Yeo\thanks{Department of Mathematics and Computer Science, University of Southern Denmark, and Department of Mathematics, University of Johannesburg.   {\tt andersyeo@gmail.com}.}
}

\begin{document}

\maketitle

\begin{abstract}
A long-standing conjecture of Caro (Discrete Math, 1994), confirmed by Ferber and Krivelevich (Adv\ Math, 2022), states that every $n$-vertex graph $G$ without isolated vertices contains an induced subgraph of order linear in $n$ in which every vertex has odd degree. We generalize this result to graphs $G$ whose vertices are labeled by $\ell: V(G)\to \{0,1\}$. We require, in an induced subgraph, all
$0$-labeled vertices to have even degree and all $1$-labeled vertices to have odd degree. Let $h_{\ell}(G)$ denote the maximum order of such a subgraph. 
Let $f_{oe}(G)=\min_{\ell} h_{\ell}(G)$ be the worst-labeling parameter. We establish a pointwise lower bound for $h_{\ell}(G)$ that immediately yields a linear lower bound in $|V(G)|$ for $f_{oe}(G)$, where $G$ has no isolated vertices. 
For an $n$-vertex connected graph, we obtain a sharp lower bound for $f_{oe}(G)$: 
$f_{oe}(G)\ge \lceil (n-1)/\mchi(G) \rceil ,$ where $\mchi(G)$ is the maximum chromatic number of a minor of $G.$ Using proved cases of Hadwiger's Conjecture, we show that for $t\in \{3,4,5,6\}$, if an $n$-vertex connected graph $G$ is $K_t$-minor-free, then $f_{oe}(G)\ge \lceil (n-1)/(t-1)\rceil$ and this bound is sharp for each $t\in \{3,4,5,6\}$. 
Finally, we conjecture that $f_{oe}(G)\ge f_o(G)/2$ for all graphs $G$ and confirm the conjecture for all trees and complete multipartite graphs.
\end{abstract}

\section{Introduction}

A graph $H$ is {\em odd} if the degree of every vertex of $H$ is odd. A key result on odd spanning subgraphs was obtained by 
Scott \cite{Scott2001} who proved that every connected graph $G$ with even number of vertices has an odd spanning forest $F$ such that
each component of $F$ is an induced subgraph of $G$. 
Gutin \cite{Gutin2016} and  Caro, Lauri and Zarb \cite{CLZ} found short proofs of
Scott's theorem. Gutin and Yeo \cite{GutinYeo2022} generalized Scott's theorem to graphs whose vertices are labeled by parities such that the sum of the parities is even.
We call a labeling $\ell:\ V(G)\to \{0,1\}$ {\em even-sum} if $\sum_{v\in V(G)}\ell(v)$ is even.

\begin{theorem}\label{thm:GY}
If a graph $G$ is connected and a labeling $\ell:\ V(G)\to \{0,1\}$ is even-sum, 
then $G$ has a spanning forest $F$ such that every component of $F$ is an induced subgraph of $G$ and 
the degree $d_F(x)\equiv \ell(x) {\pmod 2}$ for every $x\in V(F).$ 
\end{theorem}

Henceforth, an induced subgraph $H$ of $G$ is called {\em $\ell$-admissible} if $d_H(x)\equiv \ell(x) {\pmod 2}$ for every $x\in V(H).$
We also call $V(H)$ an {\em $\ell$-admissible} set. 

For odd induced subgraphs, let $f_o(G)$ denote the maximum order of an odd induced subgraph of a graph $G$.
Let $G$ be a graph of order $n$ without isolated vertices. Resolving a problem of Noga Alon, Caro \cite{Caro1994} proved that $f_o(G)=\Omega(\sqrt{n})$ and conjectured the linear lower bound $f_o(G)=\Omega(n)$. Scott \cite{Scott1992} subsequently obtained $f_o(G)=\Omega(n/\log n)$, and the conjecture was finally settled by Ferber and Krivelevich.

\begin{theorem}[Ferber-Krivelevich \cite{FK}]\label{thm: os linear bound}
Every graph of order $n$ without isolated vertices contains an odd induced subgraph of order at least $cn$, where $c=10^{-4}$.
\end{theorem}

In Section~\ref{sec:FK} we generalize this theorem to graphs whose vertices are labeled by parities. 
Given a labeling $\ell:\ V(G)\to\{0,1\}$, write $V_i=\ell^{-1}(i)$ and define
$$
h_{\ell}(G):=\max\bigl\{\,|U|:\ U\subseteq V(G), U \mbox{ is $\ell$-admissible}\bigr\}.
$$
We also introduce the worst-labeling parameter
$$
f_{oe}(G):=\min_{\ell: V(G)\to\{0,1\}} h_{\ell}(G),
$$
which measures the largest $\ell$-admissible induced subgraph guaranteed against an adversarial labeling. 
Our main result, Theorem~\ref{oddevenIS}, is a point-wise lower bound on $h_{\ell}(G)$ expressed in terms of $|V_0|$, the number of isolated vertices of $G[V_1]$, and $f_o(G[V_1])$. Combined with Theorem~\ref{thm: os linear bound}, it yields the universal linear bound $f_{oe}(G)\ge |V(G)|/10004$ (Corollary~\ref{cor:universal}).

To prove Theorem~\ref{oddevenIS}, we will also use the following
classical theorem of Gallai (\cite{L1993}, Problem 5.17). 

\begin{theorem}[Gallai's Theorem]\label{thm: gallai thm}
For every graph $G$, the vertex set $V(G)$ admits
\begin{enumerate}
    \item a partition $V(G)=V_1\cup V_2$ such that all degrees in both $G[V_1]$ and $G[V_2]$ are even;
    \item a partition $V(G)=V_1\cup V_2$ such that all degrees in $G[V_1]$ are odd and all degrees in $G[V_2]$ are even.
\end{enumerate}
\end{theorem}

\medskip

Scott \cite{Scott1992} proved that $f_o(G)\ge |V(G)|/(2\chi(G))$ for every graph $G$ without isolated vertices. He posed the following:
\begin{conjecture}\label{conj:scott}
For every graph $G$ without isolated vertices, $f_o(G)\ge |V(G)|/\chi(G).$   
\end{conjecture}
Wang and Wu \cite{WW} showed that Conjecture \ref{conj:scott} fails for bipartite graphs, but holds for line graphs. It is still unknown whether the conjecture holds for 
non-bipartite graphs.
Very recently, Ning \cite{Ning} extended the positive result in \cite{WW} by confirming Conjecture \ref{conj:scott} for claw-free graphs without isolated vertices. 

In this paper, we will use the following parameter related to $\chi(G)$ and introduced by Halin \cite{Halin} who called it {\em modified chromatic number}.
For a graph $G$, let
$$
  \mchi(G)=\max\{\chi(H): H \text{ is a minor of } G\}
$$
Barioli et al. \cite{Barioli+} extended this definition to other graph parameters. 
Hogben \cite{Hogben} is an informative survey on the topic.

The main result of Section \ref{sec:trees} is Theorem \ref{thm:tw} whose proof uses Theorem \ref{thm:GY} thus providing a novel link between spanning and induced $\ell$-admissible subgraphs. The main part of Theorem \ref{thm:tw}
is the following sharp bound on $f_{oe}(G)$ for every connected graph $G$ of order $n$: $f_{oe}(G)\ge \lceil (n-1)/\mchi(G) \rceil$. 
This bound can be useful for classes $\cal C$ of connected graphs $G$, where $\mchi(G)$ is bounded by a constant. We show that $\mchi$ is bounded in $\cal C$ by a constant if and only if $\cal C$ excludes some fixed complete graph $K_t$ as a minor. 

Using proved cases of Hadwiger's Conjecture, we show that for $t\in \{3,4,5,6\}$, if an $n$-vertex connected graph $G$ is $K_t$-minor-free, then $f_{oe}(G)\ge \lceil (n-1)/(t-1)\rceil$ and this bound is sharp for every
$t\in \{3,4,5,6\}$. In particular, for connected $K_4$-minor-free graphs, we have  $f_{oe}(G)\ge \lceil (n-1)/3\rceil .$ It is different from the bound
$f_o(G)\ge 2n/5$ for $K_4$-minor-free graphs without isolated vertices proved by Hou et al. \cite{HYLL}.
Also, for a connected planar $n$-vertex graph $G$, we have $f_{oe}(G)\ge \lceil (n-1)/4 \rceil$ and this bound is sharp. Indeed, planar graphs are $K_5$-minor-free and the sharpness is due to $K_{2,2,2}$, see the case of $t=5$ in Theorem \ref{thm:Hadw}. As far as we know, there was no sharp lower bound for $f_o(G)$, where $G$ is a connected planar graph. The bound $f_o(G)\ge n/(2\chi(G))$ of \cite{Scott1992} implies only that $f_o(G)\ge n/8$. 

The second part of Theorem \ref{thm:tw} is the claim that for every connected graph $G$ of order $n$ and an even-sum labeling $\ell$,
we have $h_{\ell}(G)\ge \lceil n/\mchi(G) \rceil$. This, in particular, allows us to show that every connected chordal graph with even number of vertices satisfies Conjecture \ref{conj:scott}.
It would be interesting to know whether every connected chordal graph with odd number of vertices satisfies Conjecture \ref{conj:scott}.

In Section \ref{sec:conj}, we conjecture that $f_{oe}(G)\ge f_o(G)/2$ for all graphs $G$ and prove the inequality for all trees and complete multipartite graphs.



\section{Generalizing Ferber--Krivelevich Theorem}\label{sec:FK}

We use the following theorem of Scott \cite{Scott1992}.

\begin{theorem}[Scott \cite{Scott1992}]\label{thm:Scott}
Let $G$ be a graph without isolated vertices and with a maximum independent set of size $p$. Then $G$ contains an odd induced subgraph of order at least $p/2$.
\end{theorem}

We extend the notation $f_o(H)$ to arbitrary graphs $H$ (possibly with isolated vertices) by defining it as the maximum order of an odd induced subgraph of $H$; isolated vertices are excluded since no odd induced subgraph can contain one.

The main result of this section is the following pointwise lower bound.

\begin{theorem}\label{oddevenIS}
Let $G=(V,E)$ be a graph without isolated vertices and let $\ell:\ V\to\{0,1\}$. Set $V_i=\ell^{-1}(i)$ for $i\in\{0,1\}$, and let $I$ be the set of isolated vertices of $G[V_1]$. Then
$$
h_{\ell}(G)\ge \max\Bigl\{\frac{|V_0|}{2},\ \frac{|I|}{2},\ f_o(G[V_1])\Bigr\}.
$$
\end{theorem}

\begin{proof}
We give three constructions.

\medskip
\noindent\textbf{Construction A: use the $0$-labeled vertices.}
Apply Theorem~\ref{thm: gallai thm}(1) to the induced subgraph $G[V_0]$. We obtain a partition $V_0=A\cup B$ such that both $G[A]$ and $G[B]$ have all degrees even. One part has size at least $|V_0|/2$; assume $|B|\ge|A|$. Then every vertex of $G[B]$ has even degree and every vertex of $B$ is labeled 0, so
$$
h_{\ell}(G)\ge|B|\ge\frac{|V_0|}{2}.
$$

\medskip
\noindent\textbf{Construction B: use the $1$-labeled vertices.}
Any odd induced subgraph of $G[V_1]$ is $\ell$-admissible in $G$ and so $h_{\ell}(G)\ge f_o(G[V_1])$.

\medskip
\noindent\textbf{Construction C: use the isolated vertices of $G[V_1]$.}
If $I=\varnothing$ the bound $|I|/2=0$ is trivial. So assume $I\neq\varnothing$. Consider $G'':=G[V_0\cup I]$. Every $u\in I$ has a neighbor in $V_0$ (since $G$ has no isolated vertices and vertices of $I$ have no neighbors in $V_1$); also $I$ is independent in $G''$.

Let $D\subseteq V_0$ be an inclusion-minimal subset that dominates $I$ in $G''$. We have the following: 

\begin{claim}\label{clm:private}
For every $w\in D$ there exists $u_w\in I$ with $N_{G''}(u_w)\cap D=\{w\}$.
\end{claim}
\begin{proof}
If no such $u_w$ existed, every neighbor of $w$ in $I$ would have another neighbor in $D\setminus\{w\}$, so $D\setminus\{w\}$ would still dominate $I$, contradicting minimality.
\end{proof}

Let $I_D:=\{u_w:w\in D\}$. Choose $D'\subseteq D$ uniformly at random. Define
$$
I_0:=\{u\in I\setminus I_D:\ |N(u)\cap D'|\text{ is odd}\}.
$$
For each $u\in I\setminus I_D$, fix $w_u\in N(u)\cap D$; the involution $S\mapsto S\triangle\{w_u\}$ shows that exactly half of all subsets of $D$ give odd intersection with $N(u)$, so
$$
\mathbb{P}\bigl(|N(u)\cap D'|\text{ is odd}\bigr)=\tfrac{1}{2}.
$$
Define
$$
I_1:=\{u_w\in I_D:\ w\in D'\text{ and }\deg_{G[D'\cup I_0]}(w)\text{ is odd}\},
$$
and set $U:=D'\cup I_0\cup I_1$.

\emph{Parity verification.} For each $u\in I_0$, all neighbors of $u$ that lie in $U$ belong to $D'$ (since $I$ is independent), 
so $\deg_{G[U]}(u)=|N(u)\cap D'|$, which is odd by definition. Each $u_w\in I_1$ has $N(u_w)\cap D\cap U=\{w\}$ and 
no neighbors in $I$, so $\deg_{G[U]}(u_w)=1$, which is odd. For $w\in D'\subseteq V_0$, if $u_x\in I_1$ 
and $x\neq w$ then $w\notin N(u_x)\cap D$ by Claim~\ref{clm:private}, so
$$
\deg_{G[U]}(w)\equiv \deg_{G[D'\cup I_0]}(w)+\mathbf{1}_{\{u_w\in I_1\}}\pmod{2},
$$
and the definition of $I_1$ ensures that this is even. Thus, $\deg_{G[U]}(v)\equiv \ell(v)\pmod{2}$ for all $v\in U$.

\emph{Size bound.} By linearity of expectation,
$$
\mathbb{E}|U|\ge\mathbb{E}|D'|+\mathbb{E}|I_0|=\frac{|D|}{2}+\frac{|I\setminus I_D|}{2}\ge
\frac{|D|}{2}+\frac{|I|-|D|}{2}=\frac{|I|}{2}.
$$
Hence, there exists a choice of $U$ with $|U|\ge |I|/2$, giving $h_{\ell}(G)\ge|I|/2$.

Combining Constructions A, B, and C completes the proof.
\end{proof}

The point-wise theorem yields the following corollaries.

\begin{corollary}\label{cor:Scott-bound}
Under the assumptions of Theorem~\ref{oddevenIS},
$$
h_{\ell}(G)\ge \frac{1}{2}\max\Bigl\{|V_0|,\ |I|,\ \alpha(G[V_1]-I)\Bigr\},
$$
where $\alpha(\cdot)$ denotes the independence number.
\end{corollary}
\begin{proof}
By Theorem~\ref{thm:Scott}, $G[V_1]-I$ has an odd induced subgraph of order at least $\alpha(G[V_1]-I)/2$. Since $f_o(G[V_1])=f_o(G[V_1]-I)$, the result follows from Theorem~\ref{oddevenIS}.
\end{proof}

\begin{corollary}\label{cor:universal}
Suppose every graph $H$ without isolated vertices satisfies $f_o(H)\ge c|V(H)|$ for some $c>0$. Then every graph $G$ without isolated vertices and every labeling $\ell:\ V(G)\to\{0,1\}$ satisfy
$$
h_{\ell}(G)\ge \frac{c}{1+4c}\,|V(G)|.
$$
In particular, Theorem~\ref{thm: os linear bound} implies $f_{oe}(G)\ge \frac{1}{10004}\,|V(G)|$.
\end{corollary}
\begin{proof}
Let $n=|V(G)|$, $x=|V_0|/n$, $y=|I|/n$, and $\kappa=c/(1+4c)$. Since $G[V_1]-I$ has no isolated vertices, $f_o(G[V_1])\ge c(1-x-y)n$. Suppose for contradiction that all three quantities in Theorem~\ref{oddevenIS} are less than $\kappa n$. Then $x<2\kappa$, $y<2\kappa$, and $1-x-y<\kappa/c$. Adding the three inequalities gives us $1<4\kappa+\kappa/c=(4c+1)/(1+4c)=1$, a contradiction.
\end{proof}

\section{Bounding $h_{\ell}(G)$ for connected graphs using modified chromatic number}\label{sec:trees}

Recall that $\mchi(G)=\max\{\chi(H): H \text{ is a minor of } G\}.$


\begin{theorem}\label{thm:tw}
Let $G$ be a connected graph of order $n$.
Then \begin{equation}\label{ineq:1} f_{oe}(G)\ge  \lceil (n-1)/\mchi(G) \rceil\end{equation}  
If a labeling $\ell$ of $V(G)$ is even-sum, then \begin{equation}\label{ineq:2}h_{\ell}(G)\ge  \lceil n/\mchi(G) \rceil\end{equation}
Both bounds are sharp.
\end{theorem} 
\begin{proof}
Let $\ell:\ V(G)\to\{0,1\}$ be an arbitrary labeling.  Consider two cases.

\2

\noindent {\bf Case 1: $\ell$ is even-sum.} By Theorem \ref{thm:GY}, $V(G)$ can be partitioned into sets $X_1,\dots ,X_p$ such that each $G[X_i]$ is connected and $\ell$-admissible.  We will call the subgraphs $G[X_i]$ {\em $\ell$-subgraphs}. Let $G^*$ be obtained from $G$ by contracting $X_1,\dots ,X_p$ and let $x_1,\dots ,x_p$ be the corresponding vertices of $G^*.$ Let $c$ be a proper coloring of $G^*$ into $\chi(G^*)$ colors and 
let the color of $\ell$-subgraph $G[X_i]$ be the color of $x_i$ in $c$ for each $i\in [p].$ Observe that there is no edge between any $\ell$-subgraphs of the same color. 
Thus, the union of all $\ell$-subgraphs of the same color is $\ell$-admissible. 

Hence, $V(G)$ can be partitioned into $\chi(G^*)$ sets each of which induces an $\ell$-admissible subgraph. Let $s$ be the maximum size of such a set. Clearly, $s\ge \lceil n/\chi(G^*) \rceil\ge \lceil n/\mchi(G)\rceil.$ 
Thus, (\ref{ineq:2}) holds.

\2

\noindent {\bf Case 2: $\ell$ is not even-sum.} Let $U:=\ell^{-1}(1)=\{v\in V(G): \ell(v)=1\}$ and let $T$ be a spanning tree of $G$. We will show that there is a vertex $u\in U$ such that every connected component of $T-u$ contains an even number of vertices from $U$.
If $|U|=1$, take $u$ to be the unique vertex in $U$. Then every connected component of $T-u$ contains no vertex from $U$. Now suppose that $|U|\ge 3$. Let $T_U$ be the smallest connected subtree of $T$ containing $U$. Choose any leaf $u$ of $T_U.$
By minimality of $T_U$, every leaf of $T_U$ must belong to $U$. Hence $u\in U$.
Since $u$ is a leaf of $T_U$, among the connected components of $T-u$, at most one component contains vertices from $U\setminus\{u\}$. In fact, that component contains all of $U\setminus\{u\}$, while the other components contain no vertices from $U$. Thus, the numbers of vertices from $U$ in the connected components of $T-u$ are
\(
|U|-1,\ 0,\ 0,\ldots,0.
\)
Therefore, every connected component $C$ of $T-u$ satisfies
\(
\sum_{v\in V(C)} \ell(v)\equiv 0\pmod 2.
\)

Let $C'_1,\dots ,C'_q$ be connected components of $G-u$. 
Observe that for each $i\in [q]$, $V(C'_i)$ is the union of the vertex sets of some connected components of $T-u$.
Thus, we have \(
\sum_{v\in V(C_i')} \ell(v)\equiv 0\pmod 2
\) for each $i\in [q]$. 
We can apply Theorem \ref{thm:GY} to each connected component of $G-u$ and conclude that for each $i\in [q]$, as in Case 1,
there is a set $U_i\subseteq V(C'_i)$ such that $G[U_i]$ is $\ell$-admissible in $C'_i$ and $|U_i|\ge \lceil |V(C'_i)|/\mchi(C'_i) \rceil.$
Observe that $G[U_1\cup \dots \cup U_q]$ is $\ell$-admissible in $G$.
Since $\mchi(C'_i)\le \mchi(G)$, we have $$|U_1\cup \dots \cup U_q|\ge \sum_{i=1}^q \lceil |V(C'_i)|/\mchi(G) \rceil\ge \lceil (n-1)/\mchi(G)\rceil.$$ Thus, $h_{\ell}(G)\ge \lceil (n-1)/\mchi(G)\rceil .$

Since $\ell$ is arbitrary, (\ref{ineq:1}) follows from the above bound and (\ref{ineq:2}). 

To see that both bounds are sharp consider $K_2$ with both vertices labeled by 0. 
\end{proof}

For each of the two bounds of Theorem \ref{thm:tw} there is a number of interesting consequences. We will start with bound (\ref{ineq:1}) and show how to use 
it to obtain a bound of the form $f_{oe}(G)\ge (n-1)/c,$ where $c$ is a constant. We can use the following:

\begin{proposition}\label{prop:finite}
Let $\cal C$ be a class of graphs. Then $\mchi$ is bounded in $\cal C$ by a constant if and only if $\cal C$ excludes some fixed complete graph $K_t$ as a minor.    
\end{proposition}
\begin{proof}
If $\mchi(G)\le k$ then $G$ cannot have a $K_{k+1}$ minor as $\chi(K_{k+1})=k+1.$ Conversely, if every $G\in \cal C$ is $K_t$-minor-free, 
then every minor of every $G\in \cal C$ is also $K_t$-minor-free. Delcourt and Postle \cite{DP} proved that every $K_t$-minor-free graph has chromatic number bounded by 
$O(t\log\log t).$ Hence, $\mchi(G)=O(t\log\log t).$
\end{proof}

Unfortunately, Proposition \ref{prop:finite} does not give us the optimal value of $c$ in $f_{oe}(G)\ge (n-1)/c.$ To get it, we can use
Hadwiger's Conjecture: $\chi(G)\le \eta(G)$ for each graph $G$, where $\eta(G)$ is the maximum $t$ such that $K_t$ is a minor of $G$. Observe that Hadwiger's Conjecture is equivalent to $\mchi(G)=\eta(G)$ for every $G$. Indeed, $K_{\eta(G)}$ is a minor of $G$ so $\mchi(G)\ge \eta(G).$ Conversely, if Hadwiger's Conjecture holds, then every minor $H$ of $G$ satisfies $\chi(H)\le \eta(H)\le \eta(G).$ Thus, $\mchi(G)\le \eta(G).$ Hadwiger's Conjecture was proved for $\eta(G)\le 5$ \cite{Dies}. Thus, we have the following:

\begin{theorem}\label{thm:Hadw}
For $t\in \{3,4,5,6\}$, if an $n$-vertex connected graph $G$ is $K_t$-minor-free, then $f_{oe}(G)\ge \lceil (n-1)/(t-1)\rceil$ and this bound is sharp. 
\end{theorem}
\begin{proof}
Since the bound itself was proved before this theorem, it is enough to prove sharpness.

\2

{\bf $t=3$ (trees):} Observe that the bound becomes an equality already for $K_2$; label both vertices by 0. 

\2 

{\bf $t=4$ (partial 2-trees):} Observe that the bound becomes an equality already for $K_3$ with labels 0, 0, 1. 

\2 

{\bf $t=5$ ($K_5$-minor-free graphs):} The bound becomes an equality for $K_{2,2,2}$. Indeed, $f_{oe}(K_{2,2,2})\ge 2$ and $f_o(K_{2,2,2})=2.$ To show that $f_o(K_{2,2,2})=2$, consider a maximum-order induced odd subgraph $H$ of $K_{2,2,2}$ and let $x_1,x_2,x_3$ be the number of vertices of $H$ in each partite set of $K_{2,2,2}$. If $\min\{x_1,x_2,x_3\}=0$ then 
we observe that $f_o(K_{2,2})=2.$ Otherwise, for each $i\in [3]$, we have $x_i>0$ and $|V(H)|-x_i$ is odd. Now, if $|V(H)|$ is odd,
then each $x_i=2$, a contradiction with $|V(H)|$ being odd. If $|V(H)|$ is even then each $x_i=1$, a contradiction with $|V(H)|$ being even.
Note that $K_{2,2,2}$ is $K_5$-minor-free as it is planar. 

\2

{\bf $t=6$ ($K_6$-minor-free graphs):} Let $C_7=v_0v_1v_2\dots v_6v_0$ and consider $G=\comp{C_7}.$ We will prove that $G$ is $K_6$-minor-free and $f_o(G)=2=\lceil 6/5\rceil.$
Suppose that $G$ has a $K_6$ minor. Since $G$ has only
seven vertices, a $K_6$-minor model either has six singleton branch sets,
which would give a proper $K_6$ subgraph, or has one branch set of size $2$ and five
singleton branch sets.  In the latter case, the five singleton branch sets
would have to be pairwise adjacent, giving a $K_5$ subgraph.  This is
impossible because the maximum clique $\omega(G)=3$.  Therefore $G$ is $K_6$-minor-free.

Now we prove that $f_o(G)=2$. Let $G[S]$ be an odd subgraph. Clearly, $f_o(G)\ge 2$ and 
since $|S|$ is even, to complete the proof we should exclude the cases $|S|=4$ and $|S|=6$. 
For each $v\in S$,
\[
        d_{G[S]}(v)=|S|-1-d_{C_7[S]}(v).
\]
If $|S|=4$ or $|S|=6$, then $|S|-1$ is odd.  Since $d_{G[S]}(v)$ must
be odd, it follows that $d_{C_7[S]}(v)$ is even for every $v\in S$.
But $C_7[S]$ is a proper induced subgraph of a cycle, hence is a forest.  A
forest in which every vertex has even degree has no edges.  Consequently $S$
would have to be an independent set in $C_7$.  This is impossible for
$|S|=4$ or $|S|=6$, because the largest independent sets of $C_7$ have only 3 vertices each.
\end{proof}

Bound (\ref{ineq:2}) implies $f_o(G)\ge  \lceil n/\mchi(G) \rceil$ for a connected graph $G$ with even number of vertices. We will show that this bound can be used for constructing graph classes satisfying Conjecture \ref{conj:scott}. Recall that Conjecture \ref{conj:scott} claims that $f_o(G)\ge n/\chi(G)$ for every graph $G$ without isolated vertices and it was proved for claw-free graphs without isolated vertices \cite{Ning}. Hence, we are interested in graph classes that include some non-claw-free graphs. 

\begin{corollary}\label{cor:Scott}
Let $\cal C$ be a class of connected graphs with even number of vertices
such that for every graph $G$ in $\cal C$ we have $\mchi(G)=\chi(G)$. 
Then all graphs in $\cal C$ satisfy Conjecture \ref{conj:scott}.    
\end{corollary}

Now we will show that every connected chordal graph $G$ with even number of vertices satisfies Corollary \ref{cor:Scott}. As chordal graphs are perfect and their tree-width equals the order of maximum clique minus 1, we have $\chi(G)=\omega(G)=\tw(G)+1.$
For every graph $L$, $\chi(L)\le \tw(L)+1$ \cite{Dies}. Hence, for a minor $H$ of $G$, we have $\chi(H)\le \tw(H)+1\le \tw(G)+1=\chi(G)$. Hence $\mchi(G)=\chi(G)$ and so $G$ satisfies Conjecture \ref{conj:scott}. To construct a non-claw-free chordal graph
start from a complete graph $K_p$, choose $x\in V(K_p)$ and add four new vertices adjacent to $x$.


\section{Conjecture on lower bound for $f_{oe}(G)/f_o(G)$}\label{sec:conj}

We state the following conjecture and provide two results supporting it. 

\begin{conjecture}\label{conj:half}
For every graph $G$, we have $f_{oe}(G)\ge f_o(G)/2$.
\end{conjecture}



\begin{proposition}\label{prop:tree}
For every tree $T$, we have $f_{oe}(T)\ge f_o(T)/2.$
\end{proposition}
\begin{proof}
Let $T$ have $n$ vertices and $\ell$ be an arbitrary labeling of $V(T)$.
By (\ref{ineq:1}),
$h_{\ell}(T) \ge \lceil{\frac{n-1}{2}}\rceil = \lfloor{\frac{n}{2}}\rfloor$.  If $T$ itself is odd, then $n$ is even and
$f_o(T)=n$, so
$h_{\ell}(T) \ge \frac{n}{2} = \frac{f_o(T)}{2}.$
If $T$ is not odd, then $f_o(T)\le n-1$. Note that $h_{\ell}(T) \ge \lfloor{\frac{n}{2}}\rfloor\ge \frac{n-1}{2}.$ 
Hence, $h_{\ell}(T) \ge \frac{f_o(T)}{2}.$
Since $\ell$ is arbitrary, we conclude that $f_{oe}(T)\ge f_o(T)/2$.
\end{proof}





\begin{theorem}
 If $G$ is a complete multipartite graph, then $f_{oe}(G) \geq f_o(G)/2$.
\end{theorem}
\begin{proof}
Let $U \subseteq V(G)$ such that all degrees in $G[U]$ are odd and $|U|=f_o(G)$.  
Let $V_1,V_2,\ldots,V_p$ be the partite sets of $G[U]$. 
Note that $|U|$ is even as every vertex degree in $G[U]$ is odd. Since the degree in $G[U]$ of every vertex in $V_i$ is $|U|-|V_i|,$ which is odd, $|V_i|$ is odd for every $i \in [p]$. 
Since $|U|$ is even and each $|V_i|$ is odd, $p$ is even. 

Let $\ell : V(G) \rightarrow \{0,1\}$ be a labeling of $V(G)$ and let $V_i^j=\ell^{-1}(j)\cap V_i,$ where $j\in \{0,1\}.$
Define $r_1,r_2,\ldots,r_p$ such that $|V_i^{r_i}| > |V_i^{1-r_i}|$ for every $i \in [p]$, which is possible as $|V_i|$ is odd.
Let $W = \cup_{i=1}^p V_i^{r_i}$ and let $V_i^W = V_i \cap W$ for all $i \in [p]$.
As $|V_i^{r_i}| \geq |V_i^{1-r_i}|+1$, we have $|W| \geq (|U| + p)/2$.

Note that all vertices in $V_i^W$ have the same degree and also the same $\ell$-value. We call $V_i^W$ {\em good} if the degree and $\ell$-value are of the same parity for every vertex in $V_i^W$. Otherwise, we call $V_i^W$ {\em bad}.
Without loss of generality, assume that all $V_1^W,V_2^W,\ldots , V_s^W$ are bad and all $V_{s+1}^W,V_{s+2}^W,\ldots , V_p^W$ are good. 
We consider the following cases, where in Cases 1-3, $s$ is even and in Cases 4-6, $s$ is odd. 

\2

{\bf Case 1:} $s$ is even and $s \leq p/2$. We remove one vertex from each set $V_1^W,V_2^W,\ldots , V_s^W$ in $W$ and denote the resulting set $W^*$. Observe that $|W^*| \geq |W|- p/2 \geq |U| / 2$. Also, $W^*$ is $\ell$-admissible. Indeed, if $x\in W^*$ is in a good set, $d_{G[W^*]}(x)$ and $d_W(x)$ are of the same parity  since $s$ is even, implying that $d_{G[W^*]}(x)$ and $\ell(x)$ are of the same parity, and if $x\in W^*$ is in a bad set, $d_{G[W^*]}(x)$ and $d_{G[W]}(x)$ are of different parity since $s-1$ is odd, which means that $d_{W^*}(x)$ and $\ell(x)$ are of the same parity. 

\2

{\bf Case 2:} $s$ is even, $s > p/2$ and  $|V_i^{r_i}| \geq |V_i^{1-r_i}|+2$ for all $i \in [s]$. 
We remove one vertex from each set $V_1^W,V_2^W,\ldots , V_s^W$ in $W$ and denote the resulting set $W^*$.
We then have $|W^*| \geq |U| / 2$, as $W^*$ contains at least half of the vertices in $V_i$ for all $i \in [p]$. 
As in Case 1, $W^*$ is $\ell$-admissible.

\2

{\bf Case 3:} $s$ is even and $s > p/2$ and $|V_{i_0}^{r_{i_0}}| = |V_{i_0}^{1-r_{i_0}}|+1$ for some $i_0 \in [s]$. Since $p$ is even, $s \ge p/2+1$ and $p-s \le p/2-1$.
In $W$, we remove one vertex from each set $V_{s+1}^W,V_{s+2}^W,\ldots , V_p^W$ and replace  $V_{i_0}^{r_{i_0}}$ with $V_{i_0}^{1-r_{i_0}}$. We denote the resulting set $W^*$.
We have $|W^*| \geq |W| - (p/2-1)-1 \geq |U| / 2$. Note that for a vertex $x$ in $G[U]$ removing one vertex from another partite set flips parity the degree of $x$; 
removing one from the same partite set does not; replacing $V_{i_0}^{r_{i_0}}$ with $V_{i_0}^{1-r_{i_0}}$  changes the label and changes selected size by one in the margin-one case.
Hence, $W^*$ is $\ell$-admissible.

\2

{\bf Case 4:} $s$ is odd and $s \geq p/2$. We remove one vertex from each set $V_{s+1}^W,V_{s+2}^W,\ldots , V_p^W$ in $W$ and denote the resulting set $W^*$. We note that again $W^*$ is $\ell$-admissible and $|W^*| \geq |W| - p/2 \geq |U|/ 2$.

\2

{\bf Case 5:} $s$ is odd and $s < p/2$ and  $|V_i^{r_i}| \geq |V_i^{1-r_i}|+2$ for all $i \in \{s+1,s+2,\ldots,p\}$. 
We remove one vertex from each set $V_{s+1}^W,V_{s+2}^W,\ldots , V_p^W$ in $W$ and denote the resulting set $W^*$.
We then have $|W^*| \geq |U| / 2$, as $W^*$ contains at least half of the vertices in $V_i$ for all $i \in [p]$. 
As before, $W^*$ is $\ell$-admissible.

\2

{\bf Case 6:} $s$ is odd, $s < p/2$ and $|V_{i_0}^{r_{i_0}}| = |V_{i_0}^{1-r_{i_0}}|+1$ for some ${i_0} \in \{s+1,s+2,\ldots,p\}$. Since $p$ is even, $s \le p/2-1$.
Then remove one vertex from each set $V_{1}^W,V_{2}^W,\ldots , V_s^W$ and replace $V_{i_0}^{r_{i_0}}$ with $V_{i_0}^{1-r_{i_0}}$. We denote the resulting set $W^*$.
Then $|W^*| \geq |W| - (p/2-1)-1 \geq |U| / 2$ and, similar to Case 3, $W^*$ is $\ell$-admissible. 

Since $\ell$ is arbitrary, we conclude that $f_{oe}(G) \geq f_o(G)/2$.
\end{proof}


\end{document}